\documentclass[11pt]{article}
\usepackage{amsmath, amssymb, amsthm, mathrsfs}
\usepackage{geometry}
\usepackage{tikz}
\usetikzlibrary{arrows.meta,decorations.pathmorphing,patterns}

\newtheorem{theorem}{Theorem}
\newtheorem{lemma}[theorem]{Lemma}
\newtheorem{prop}[theorem]{Proposition}

\newcommand{\N}{\mathbb{N}}

\newcommand{\rep}[2]{r_{#1}(#2)} 
\newcommand{\rept}[2]{\tilde{r}_{#1}(#2)} 

\renewcommand{\S}{\mathscr{S}}

\title{Three Questions of Erd\H{o}s-Nathanson on Asymptotic Bases of Order 2}
\author{Daniel Larsen}
\date{}
\begin{document}

\maketitle

\begin{abstract}
    We study three natural properties that measure the robustness of asymptotic bases of order 2: having divergent representation function, being decomposable as a union of two bases, and containing a minimal basis. Erd\H{o}s and Nathanson showed that sufficiently rapid growth of the representation function (specifically, $r_A(n) \ge C \log n$ for appropriate $C$) implies both decomposability and the existence of a minimal basis. We prove that for weaker growth rates, these three properties are independent. The construction proceeds via an inductive scheme on exponentially growing intervals.
\end{abstract}

\section{Introduction}
A set $A \subseteq \N$ is an \textit{asymptotic basis} (always of order 2 in this paper) if there exists $n_0$ such that every $n > n_0$ can be written as $n = a + a'$ with $a\le a' \in A$. The \textit{representation function} $\rep{A}{n}$ counts the number of such representations. 
We define $\rept{A}{n}$ to be the number of representations $a+a'$ with $a'/a\in [1,100]$.
A subset $B \subseteq A$ is a \textit{minimal asymptotic basis} if $B$ is an asymptotic basis, but no proper subset of $B$ is an asymptotic basis.

We consider three natural properties of a basis:
\begin{enumerate}
    \item[(P1)] \textbf{Divergent representation function:} $\rep{A}{n} \to \infty$ as $n \to \infty$.
    \item[(P2)] \textbf{Decomposability:} $A = B \cup C$ where $B$ and $C$ are disjoint asymptotic bases.
    \item[(P3)] \textbf{Existence of minimal basis:} $A$ contains a minimal asymptotic basis.
\end{enumerate}

Properties (P1) and (P2) are both indications of robustness.
Property (P3) is less obviously of the same kind. In particular, a minimal asymptotic basis is as far from being robust as possible while clearly satisfying (P3).
However, if every integer has a sufficient number of representations in $A$, then there are many choices of asymptotic subbases, so one might expect a minimal one.
Indeed,
Erd\H{o}s and Nathanson \cite{EN88} showed that for every constant $C > (\log \frac 43)^{-1}$, if $\rep{A}{n} > C \log n$ for all sufficiently large $n$, then $A$ satisfies both (P2) and (P3). This mysterious constant suggests a linkage between the three notions. However, in this paper, we will show that
for slower growth rates, the three properties are mutually independent:

\begin{theorem} \label{thm:main}
There exist asymptotic bases satisfying all possible combinations of (P1), (P2), and (P3).
\end{theorem}

In \cite[Question 2]{EN88}, Erd\H{o}s and Nathanson asked if (P1) implies (P2). This was previously shown not to be the case by the author, 
using a multiagent system built on Claude 4.5 and Gemini 3 Pro, with a construction based on \cite{EN88}. 
In \cite{EN79}, Erd\H{o}s and Nathanson speculated 
that (P1) might not be enough to imply (P3). 
In the preprint \cite{LL26}, we used a different construction to confirm this idea, showing that even 
$r_A(n) > \epsilon \log n$ for large $n$ is not enough to imply (P3).  
Finally, \cite[Question 4]{EN88} asks whether (P2) implies (P3). This paper shows that the answer is no.

All eight cases of Theorem~\ref{thm:main} follow from a single construction.  We will need to define some auxiliary functions 
$\phi,\psi\colon \N\to\N$
which will depend on which 
properties we are trying to achieve.
To begin with, we take $h(n) := \lfloor n\cdot 10^{-8}\rfloor$.
If we want (P1) to hold, we choose $\phi(n):=n$; otherwise, we use the constant function $2$ if  (P2) and $1$ otherwise.
If we want (P3) to hold, we choose $\psi$ to be the constant function $1$; otherwise, we choose $\psi(n):=n$.

Let $N_i = 4^{i+1}$.
Let $\S$ be any selection mechanism for producing sets of integers 
$$G_k\subset \bigcup_{i=1}^{k-1} \bigl(B_i \cup (N_{i+1}-G_i)\bigr)$$
and 
$$H_k\subset \bigcup_{i=1}^{k-1} \bigl(C_i \cup (N_{i+1}-H_i)\bigr)$$
given any sets of integers $B_i,C_i,G_i,H_i$ for $1\le i\le k-1$.

\begin{theorem}
\label{construction}
If $\S$ has the property that for all $k$ and for all input sets, $G_k$ and $H_k$ are disjoint and 
no element of $F_k := G_k \cup H_k$ is greater than $N_k$, then
there exist sets $B_1,B_2,\ldots$ and $C_1,C_2,\ldots$ such that setting $G_k,H_k$ to be the outputs of $\S$ with these inputs (including $G_i,H_i$ for $i<k$), the following conditions are satisfied by
$B := \bigcup_{i=1}^{\infty} \bigl(B_i \cup (N_{i+1}-G_i)\bigr)$ and $C :=  \bigcup_{i=1}^{\infty} \bigl(C_i \cup (N_{i+1}-H_i)\bigr)$:

\begin{enumerate}
        \item Each $B_k, C_k$ consists of positive integers strictly between $N_{k}$ and $N_{k+1}$.
        \item $B \cap C = \emptyset$.
        \item  $\rept{B}{n}, \rept{C}{n} \ge h(n)$ for all sufficiently large elements of $\N\setminus\{N_i\}$.
        \item For sufficiently large $k$, all representations of $N_{k+1}$ as sums of two elements of $B\cup C$ meet $F_k$.
\end{enumerate}
\end{theorem}

We give a proof of this theorem in section 4 by means of a randomized iterative construction of 
$B(k) = B\cap [1,N_{k+1}]$ and $C(k) = C\cap [1,N_{k+1}]$. 

We conclude this section with a table summarizing the essential features of the eight constructions given in sections 2 and 3.
\begin{center}
\begin{tabular}{|cccccc|}
\hline
(P1) & (P2) & (P3) & $A$ & $\phi(n)$ & $\psi(n)$ \\
\hline
T & T & F & $B \cup C$ & $n$ & $n$ \\
F & T & F & $B \cup C$ & 2 & $n$ \\
T & T & T & $B \cup C$ & $n$ & 1 \\
F & T & T & $B \cup C$ & 2 & 1 \\
T & F & F & $B$ & $n$ & $n$ \\
F & F & F & $B$ & 1 & $n$ \\
T & F & T & $B$ & $n$ & 1 \\
F & F & T & $B$ & 1 & 1 \\
\hline
\end{tabular}
\end{center}

\section{Decomposable Case}

We now explain how Theorem \ref{construction} implies the four cases of Theorem \ref{thm:main} in which $A$ is decomposable.
In every case, it is a matter of describing the appropriate selection mechanism $\S$.

At stage $k$, then, $B_i,C_i,G_i,H_i$ are given for $i<k$, and we need to describe how to choose $G_k$ and $H_k$. 
We define $B(k-1) = \bigcup_{i=1}^{k-1} (B_i  \cup (N_{i+1}-G_i))$ and likewise for $C(k-1)$.
We say a set $S$ is \emph{$k$-eligible} if 
$S = \{s_1,\ldots,s_{\phi{(k)}}\}\subset B(k-1)\cup C(k-1)$ 
where 
\begin{equation}
\label{eligibility}
\psi(k)\le s_1<\cdots<s_{\phi(k)}\le N_k
\end{equation}
and $|S\cap B(k-1)|, |S\cap C(k-1)|\ge 1$.
A set is \emph{cumulatively} $k$-\emph{eligible} if it is $i$-eligible for some $i\le k$.

Let $k_0\in \N\cup\{\infty\}$ be minimal such that for all $k> k_0$, there are at least $k$ sets which are $k$-eligible.
Then we set  $F_k = \emptyset$ for $k\le k_0$.
Otherwise,  we choose $F_k$ with largest element minimal among all cumulatively $k$-eligible sets which have not yet occurred.
We then define $\S$ to produce
$$G_k := F_k \cap B(k-1),\ H_k := F_k \cap C(k-1).$$

\begin{lemma}
\label{k0}
The value $k_0$ is finite.
\end{lemma}

\begin{proof}
If not, $\S$ always produces empty sets.
Applying Theorem~\ref{construction}, we see that for large $k$, $B(k-1)$ has $\gg N_k^{1/2}$ elements,
and, in particular, many more than $\phi(k)$ elements which are greater than $\psi(k)$. This gives a contradiction.
\end{proof}

Every cumulatively $k$-eligible set eventually occurs as an $F_i$ for some $i$. 

Let $B$ and $C$ be defined by Theorem~\ref{construction} applied to the selection mechanism $\S$.
Set $A = B\cup C$. Note that by Property 3 of Theorem~\ref{construction},
$\rep An$ tends to $\infty$ or not according to whether $\phi$ tends to $\infty$.

We claim $B$ is a basis.
Indeed, $\rep{B}{n} \rightsquigarrow \infty$, by which we mean that 
$$\liminf_{n\in \N\setminus \{N_i\}} r_B(n)= \infty.$$
However, we also know for every large $k$, $N_{k+1} - G_k$ and $G_k$ are contained in $B(k)$, so $N_{k+1}$ has at least one representation in $B+B$.
The same argument works for $C$, and this implies decomposability of $A$.

\subsection{Bases with no minimal subbasis}
Suppose $D$ is a subbasis of $A$. Since $D+D$ contains $N_i$ for all $i$ sufficiently large, it follows that for all sufficiently large $k$ and all sets $S$ which are $k$-eligible, $D$ meets $S$.

Let $M$ be a positive integer. Let $k$ be the smallest integer with $N_k > M$. Suppose 
$$|(B\setminus D)\cap [M,100M]| \ge \phi(k+4)$$
and likewise for $C$. Then there exists a subset of $[M,100M]$ which is $k+4$-eligible
which does not meet $D$.  We conclude that for sufficiently large $M$, 
$$|(B\setminus D)\cap [M,100M]|<\phi(k+4)$$ or 
$$|(C\setminus D)\cap [M,100M]|<\phi(k+4).$$ We assume without loss of generality that the first is true.

Now delete the smallest element of $D$ to obtain $D'$. 
Fix $t\in \{0,1,2,\ldots 100\}$. Assume first that $101M-t$ is not of the form $N_i$ for any $i$.
By Property 3, $\rept{B}{101M-t} \ge h(101M-t)$ and, consequently, 
$$\rep{D\cap [M,100M]}{101M-t} > h(101M-t)-\phi(k+4) > 1.$$
Therefore, 
$$\rep{D'}{101M-t} \ge \rep{D}{101M-t}-1>0.$$
Now suppose that $101M-t = N_i$ where $i$ is sufficiently large. In this case we know that $\psi$ goes to $\infty$ and therefore for sufficiently large $k$,
the smallest element of $D$ is not contained in any $k$-eligible set. Therefore, 
$\rep{D'}{101M-t} = \rep{D}{101M-t} > 0$, making $D'$ a basis. It follows that $D$ is not minimal.

This concludes the two cases of Theorem~\ref{thm:main} in which (P2) holds but (P3) does not.

\subsection{Bases with a minimal subbasis}
Recall that in this case we set $\psi(n)=1$ for all $n$.
We claim $B$ is a minimal subbasis. Let $b\in B$ be any element. 
There exist infinitely many eligible sets containing $b$ and no other elements of $B$. This gives an infinite sequence of $N_{j_i}$ where $r_B(N_{j_i}) = 1$
and $r_{B\setminus \{b\}}(N_{j_i}) = 0$. Thus $B\setminus \{b\}$ can never be a basis, meaning that $B$ is, indeed, minimal.

\section{Indecomposable case}

Now we handle the four indecomposable cases. In this section, we say a set $S$ is \emph{$k$-eligible} if it is contained in $B(k-1)$ and satisfies \eqref{eligibility}.
With this new definition of eligibility, we usually define $\S$ in the same way as in section 2 except with $H_k=\emptyset$.

In the special case that (P3) holds but not (P1), we fix a surjective function $F\colon \N\to \N$ with all fibers infinite and such that  $\min F^{-1}(n+1) > \min F^{-1}(n)$ for all $n$.
There is some $k_0$ for which $|B(k)| > k-k_0$ for all $k$.
We then make 
$\S$ produce the sequence 
$$\underbrace{\emptyset,\ldots,\emptyset}_{k_0},\{b_{F(1)}\},\{b_{F(2)}\},\{b_{F(3)}\},\ldots,$$
where $G_k=\{b_{F(k-k_0)}\}$ for $k>k_0$, with $b_k$ the smallest element of $B(k+k_0-1)$ greater than $b_1,\ldots,b_{k-1}$.
Note that this definition is not circular because $b_{F(k-k_0)}\in B(k-1)$ since $F(k-k_0)+k_0-1 \le k-1$.

In all cases, we then apply Theorem~\ref{construction} to obtain $B$. (We also obtain $C$, but we do not use it.)
Note that by Property 3 of Theorem~\ref{construction},
$\rep Bn$ tends to $\infty$ or not according to whether $\phi$ tends to $\infty$. 
As before, $B$ is a basis, and every $k$-eligible set appears as $G_i$ for some $i$.
It follows that for all sufficiently large $k$ and all $k$-eligible sets $S$, any subbasis $D$ meets $S$.

Let $M$ be a positive integer. Let $k$ be the smallest integer with $N_k > M$. It is clear that if $D$ is a subbasis of $B$, then 
 for sufficiently large $M$,
\begin{equation}
\label{hyperdyadic}
|(B\setminus D)\cap [M,M^2]| < \phi(2k+1).
\end{equation}

Note that if $B$ decomposes as $B'\cup B''$, then this inequality holds for $D=B'$ and $D=B''$. Since $B'$ and $B''$ are disjoint,
this means that 
$$|B\cap [M,M^2]| = |B'\cap [M,M^2]| + |B''\cap [M,M^2]| < 2\phi(2k+1)\ll k,$$
which is impossible for large $M$ since $B$ is an asymptotic basis. This takes care of indecomposability in all four cases.

\begin{lemma}
\label{subbasis}
For any subbasis $D$, if $D\setminus D'$ has $o(M^2)$ elements in each hyperdyadic interval $[M,M^2]$, then $(D' + D')\cap (\N\setminus \{N_i\})$ has finite 
complement in $(D+D)\cap (\N\setminus \{N_i\}).$
\end{lemma}

\begin{proof}
Fix $0\le t\le 2M-2$. Assume that $M^2-t$ is not of the form $N_i$ for any $i$.
When $M$ is large, Theorem~\ref{construction} and \eqref{hyperdyadic} imply that
$$\rep{D\cap [M,M^2]}{M^2-t} > h(M^2-t)-\phi(2k).$$
Therefore, 
$$\rep{D'\cap [M,M^2]}{M^2-t} > h(M^2-t)-\phi(2k)-o(M^2) \gg M^2 > 0.$$

\end{proof}

\subsection{Bases with no minimal subbasis}

Let $D$ be any subbasis of $B$, and let $D'\subset D$ consist of all elements of $D$ except the smallest one.
In this case we know that $\psi$ goes to $\infty$ and therefore for sufficiently large $k$,
the smallest element of $D$ is not contained in any $k$-eligible set. Therefore, 
$\rep{D'}{N_i} = \rep{D}{N_i} > 0$ for $i$ sufficiently large. By Lemma~\ref{subbasis}, $D'$ is a basis, so $D$ is not minimal.

This concludes the two cases of Theorem~\ref{thm:main} in which neither (P2) nor  (P3) holds.

\subsection{Bases with a minimal subbasis}
Recall that in this case we set $\psi(n)=1$ for all $n$.

Let $B'$ be a subset of $B$ such that $B\setminus B'$ contains exactly $\phi(k)-1$ elements less than or equal to $N_k$ for all sufficiently large $k$
and $B'\supset N_{k+1}-G_k$ for all $k$. This is possible since $|G_k| < k$ for large $k$. If we are in the bounded representation
function case, this just means $B'=B$.

Let $b\in B'$ be any element. Then $B\setminus (B'\setminus\{b\})$ contains $\phi(k)$ elements $\le N_k$ for all large $k$, meaning that it contains infinitely many
$G_i$. Thus infinitely many $G_i$ fail to meet $B'\setminus \{b\}$, implying it is not a basis.  (Here we use the fact that $\{b\}$ appears infinitely many times as a $G_i$ in the case
of bounded representation function.)

On the other hand, it is clear that $B'$ meets $G_i$ for all sufficiently large $i$. Applying Lemma~\ref{subbasis} with $D=B$, we conclude that $B'$ is a basis
and, in fact, a minimal one.

\section{Construction}
\def\core{\mathrm{core}}

Let $X_1$, $X_2$, $X_3$ partition $\N$ in such a way that every positive integer has an equal chance of being in any of the three, and these choices are independent.
Let $X_i(S) = S\cap X_i$. 
We write $X_i(a,b)$ for $X_i((a,b))$. We also write
$(a,b)_{\core}$ for $(a+(b-a)/100,b-(b-a)/100)$.

\begin{lemma}
\label{sum}
There exists $\epsilon>0$ such that if $I$ and $J$ are intervals of positive integers of size $\ge m$, then the probability that some element of $(I+J)_{\core}$
has fewer than $m/2000$ representations in $X_i(I)+X_i(J)$ is $\ll (1-\epsilon)^m$.
\end{lemma}

\begin{proof}
Every element $z\in (I+J)_{\core}$ has at least $m/100$ representations as $x+y$, where $x\in I$ and $y\in J$.
There are therefore at least $m/200$ ways of writing $z = x_k+y_k$ where $x_k\in I$, $y_k\in J$ and the $2k$ elements 
$x_1,\ldots,x_k,y_1,\ldots,y_k$ are all distinct. For each $k$, the probability that $x_k\in X_i(I)$ and $y_k\in X_i(J)$ is $1/9$,
and these events are independent for distinct $k$. The expected number of representations is therefore at least $m/1800$,
so by the Chernoff bound, the probability that the actual number is less than $m/2000$ decays exponentially in $m$.
\end{proof}

\begin{lemma}
\label{intersection}
There exists $\epsilon > 0$ with the following property.
Let $N$ be a sufficiently large positive integer. Let $I$ and $J$ be intervals of positive integers of size $> N/100$.
If every element of $J$ is less than $4N$,  then the probability that for any 
$i,j\in \{1,2,3\}$,
$$X_i(I) + X_i(4N-X_j(J))\supset (I+4N-J)_\core\setminus \{4N\},$$
and, moreover, every element of the right hand side has greater than $N\cdot 10^{-7}$ representations as a sum of elements as in the left hand side, is greater than 
$1-(1-\epsilon)^N$.
\end{lemma}

\begin{proof}
Let $4N+\Delta$ be an arbitrary element of the right hand side.
We are looking for solutions to $x-y=\Delta$, where $x\in X_i(I)$, $y\in X_j(J)$, and $4N-y\in X_i$.

Let $S = (I-\Delta) \cap J$. 
Let $I = (a,b)$, $J=(c,d)$, and $\delta = (b-a+d-c)/100$.
Then
$$a-d + \delta  < \Delta < b - c - \delta,$$
so the infimum of $I-\Delta$ is at most $d-\delta$ and the supremum of $I-\Delta$ is at least $c+\delta$.
Consequently, either $I-\Delta \subset J$, in which case $|S| > N/100$, or $S$ contains one of $(c,c+\delta)$ or $(d-\delta,d)$, in which case $|S| \ge \delta > N/5000$.

Let $Y$ denote a maximal subset of $S$ for which the sets $Y$, $Y+\Delta$, $4N-Y$ are disjoint. 
Then $|Y|\ge  |S|/9$.
There is independent probability $1/27$ for each $y\in Y$ that $y\in X_j(J)$, $y+\Delta\in X_i(I)$, and $4N-y\in X_i$.
As in Lemma~\ref{sum}, the result follows by a Chernoff bound.
\end{proof}

We will always apply this lemma with $i=1$, $j=3$, and $J=(N/4,3N/4)$.

We now describe the construction of Theorem~\ref{construction} needed at stage $k$.
At this point, we have defined $B_i$ and $C_i$ for $i<k$ and been given $G_i$ and $H_i$ for $i \le k$.
Define
$$B(k-1) = \bigcup_{i=1}^{k-1}(B_i \cup (N_{i+1}-G_i)),\ C(k-1) = \bigcup_{i=1}^{k-1}(C_i \cup (N_{i+1}-H_i)),\ A(k-1) = B(k-1) \cup C(k-1).$$

Define $N = N_k$,
\begin{equation}
\label{B_k}
B_k =  \Bigl((4N/3,2N) \cup (8N/3,3N)\cup \bigl(4N-((0,N)\setminus A(k-1))\bigr)\Bigr) \cap X_1(0,4N),
\end{equation}
and
\begin{equation}
\label{C_k}
C_k =  \Bigl((4N/3,2N) \cup (8N/3,3N)\cup \bigl(4N-((0,N)\setminus A(k-1))\bigr)\Bigr) \cap X_2(0,4N).
\end{equation}
As usual, $B(k) = B(k-1) \cup B_k\cup (4N-G_k)$ and $C(k) = C(k-1)\cup C_k\cup (4N-H_k)$.
 Schematically, the construction looks like this:
\vskip 10pt

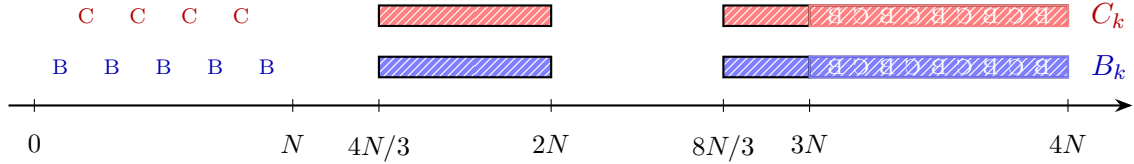
\begin{tikzpicture}[scale=1.7]
    
    \draw[-Stealth, thick] (-0.2,0) -- (8.5,0);
    
    \foreach \x/\label in {0/0, 2/N, 2.667/{4N/3},  4/{2N}, 5.333/{8N/3}, 6/{3N}, 8/{4N}} {
        \draw (\x,-0.05) -- (\x,0.05);
        \node[below, font=\small] at (\x,-0.15) {$\label$};
    }
    
    \node[font=\scriptsize, blue!70!black] at (0.2,0.3) {B};
    \node[font=\scriptsize, blue!70!black] at (0.6,0.3) {B};
    \node[font=\scriptsize, blue!70!black] at (1.0,0.3) {B};
    \node[font=\scriptsize, blue!70!black] at (1.4,0.3) {B};
    \node[font=\scriptsize, blue!70!black] at (1.8,0.3) {B};
    
\fill[blue!50!white] (2.667,0.22) rectangle (4,0.38);
\draw[thick, pattern=north east lines, pattern color=white] 
    (2.667,0.22) rectangle (4,0.38);
    
\fill[blue!50!white] (5.333,0.22) rectangle (6,0.38);
\draw[thick, pattern=north east lines, pattern color=white] 
    (5.333,0.22) rectangle (6,0.38);

\fill[blue!50!white] (6,0.22) rectangle (8,0.38);
\draw[thick, pattern=north east lines, pattern color=white] 
    (6,0.22) rectangle (8,0.38);
\draw[thick, blue!50!white] (6,0.22) rectangle (8,0.38);
    \node[font=\scriptsize, white] at (7.8,0.3) {\reflectbox{B}};
    \node[font=\scriptsize, white] at (7.6,0.3) {\reflectbox{C}};
    \node[font=\scriptsize, white] at (7.4,0.3) {\reflectbox{B}};
    \node[font=\scriptsize, white] at (7.2,0.3) {\reflectbox{C}};
    \node[font=\scriptsize, white] at (7.0,0.3) {\reflectbox{B}};
    \node[font=\scriptsize, white] at (6.8,0.3) {\reflectbox{C}};
    \node[font=\scriptsize, white] at (6.6,0.3) {\reflectbox{B}};    
    \node[font=\scriptsize, white] at (6.4,0.3) {\reflectbox{C}};
    \node[font=\scriptsize, white] at (6.2,0.3) {\reflectbox{B}};

 \node[right, blue!70!black] at (8.1,0.3) {$B_k$};

    \node[font=\scriptsize, red!70!black] at (0.4,0.7) {C};
    \node[font=\scriptsize, red!70!black] at (0.8,0.7) {C};
    \node[font=\scriptsize, red!70!black] at (1.2,0.7) {C};
    \node[font=\scriptsize, red!70!black] at (1.6,0.7) {C};
    
\fill[red!50!white] (2.667,0.62) rectangle (4,0.78);
\draw[thick, pattern=north east lines, pattern color=white] 
    (2.667,0.62) rectangle (4,0.78);
    
\fill[red!50!white] (5.333,0.62) rectangle (6,0.78);
\draw[thick, pattern=north east lines, pattern color=white] 
    (5.333,0.62) rectangle (6,0.78);

\fill[red!50!white] (6,0.62) rectangle (8,0.78);
\draw[thick, pattern=north east lines, pattern color=white] 
    (6,0.62) rectangle (8,0.78);
\draw[thick, red!50!white] (6,0.62) rectangle (8,0.78);
    \node[font=\scriptsize, white] at (7.8,0.7) {\reflectbox{B}};
    \node[font=\scriptsize, white] at (7.6,0.7) {\reflectbox{C}};
    \node[font=\scriptsize, white] at (7.4,0.7) {\reflectbox{B}};
    \node[font=\scriptsize, white] at (7.2,0.7) {\reflectbox{C}};
    \node[font=\scriptsize, white] at (7.0,0.7) {\reflectbox{B}};
    \node[font=\scriptsize, white] at (6.8,0.7) {\reflectbox{C}};
    \node[font=\scriptsize, white] at (6.6,0.7) {\reflectbox{B}};    
    \node[font=\scriptsize, white] at (6.4,0.7) {\reflectbox{C}};
    \node[font=\scriptsize, white] at (6.2,0.7) {\reflectbox{B}};

  \node[right, red!70!black] at (8.1,0.7) {$C_k$};
\end{tikzpicture}

\noindent The following diagram illustrates the probability that $n\in B(k)$, as a function of $n$.

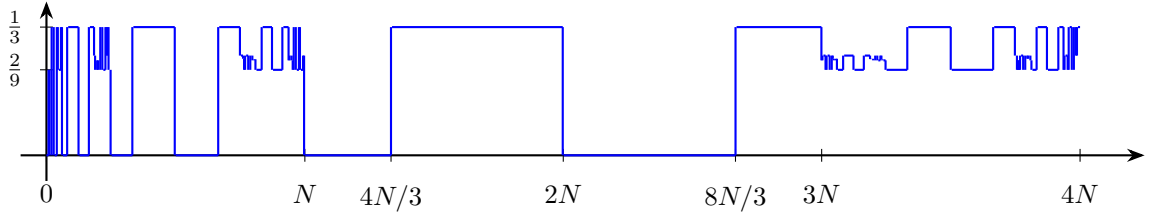
\begin{tikzpicture}[scale=1.7]
    
    \draw[-Stealth, thick] (-0.2,0) -- (8.5,0) node[right] {};
    \draw[-Stealth, thick] (0,-0.2) -- (0,1.2) node[above] {};
    
    \foreach \x/\label in {0/0, 2/N, 2.667/{4N/3},  4/{2N}, 5.333/{8N/3}, 6/{3N}, 8/{4N}} {
        \draw (\x,-0.05) -- (\x,0.05);
        \node[below, font=\small] at (\x,-0.15) {$\label$};
    }
    
    \draw (-0.05,1) -- (0.05,1);
    \node[left, font=\small] at (-0.1,1) {$\frac{1}{3}$};
    \draw (-0.05,0.667) -- (0.05,0.667);
    \node[left, font=\small] at (-0.1,0.667) {$\frac{2}{9}$};
    
    \draw[thick, blue] (0.0000,0.0000) -- (0.0160,0.0000);
    \draw[thick, blue] (0.0160,0.0000) -- (0.0160,0.6667);
    \draw[thick, blue] (0.0240,0.6667) -- (0.0240,0.0000);
    \draw[thick, blue] (0.0400,0.0000) -- (0.0400,1.0000);
    \draw[thick, blue] (0.0561,1.0000) -- (0.0561,0.0000);
    \draw[thick, blue] (0.0641,0.0000) -- (0.0801,0.0000);
    \draw[thick, blue] (0.0801,0.0000) -- (0.0801,1.0000);
    \draw[thick, blue] (0.0881,1.0000) -- (0.0881,0.6667);
    \draw[thick, blue] (0.0961,0.6667) -- (0.1121,0.6667);
    \draw[thick, blue] (0.1121,0.6667) -- (0.1121,1.0000);
    \draw[thick, blue] (0.1201,1.0000) -- (0.1201,0.0000);
    \draw[thick, blue] (0.1281,0.0000) -- (0.1602,0.0000);
    \draw[thick, blue] (0.1602,0.0000) -- (0.1602,1.0000);
    \draw[thick, blue] (0.1682,1.0000) -- (0.2482,1.0000);
    \draw[thick, blue] (0.2482,1.0000) -- (0.2482,0.0000);
    \draw[thick, blue] (0.2563,0.0000) -- (0.3283,0.0000);
    \draw[thick, blue] (0.3283,0.0000) -- (0.3283,1.0000);
    \draw[thick, blue] (0.3363,1.0000) -- (0.3684,1.0000);
    \draw[thick, blue] (0.3684,1.0000) -- (0.3684,0.7778);
    \draw[thick, blue] (0.3764,0.7778) -- (0.3764,0.6667);
    \draw[thick, blue] (0.3924,0.6667) -- (0.3924,0.7407);
    \draw[thick, blue] (0.4004,0.7407) -- (0.4004,0.6667);
    \draw[thick, blue] (0.4164,0.6667) -- (0.4164,1.0000);
    \draw[thick, blue] (0.4324,1.0000) -- (0.4324,0.6667);
    \draw[thick, blue] (0.4404,0.6667) -- (0.4565,0.6667);
    \draw[thick, blue] (0.4565,0.6667) -- (0.4565,1.0000);
    \draw[thick, blue] (0.4645,1.0000) -- (0.4645,0.7778);
    \draw[thick, blue] (0.4725,0.7778) -- (0.4725,1.0000);
    \draw[thick, blue] (0.4805,1.0000) -- (0.4805,0.6667);
    \draw[thick, blue] (0.4965,0.6667) -- (0.4965,0.0000);
    \draw[thick, blue] (0.5045,0.0000) -- (0.6647,0.0000);
    \draw[thick, blue] (0.6647,0.0000) -- (0.6647,1.0000);
    \draw[thick, blue] (0.6727,1.0000) -- (0.9930,1.0000);
    \draw[thick, blue] (0.9930,1.0000) -- (0.9930,0.0000);
    \draw[thick, blue] (1.0010,0.0000) -- (1.3293,0.0000);
    \draw[thick, blue] (1.3293,0.0000) -- (1.3293,1.0000);
    \draw[thick, blue] (1.3373,1.0000) -- (1.4975,1.0000);
    \draw[thick, blue] (1.4975,1.0000) -- (1.4975,0.7778);
    \draw[thick, blue] (1.5055,0.7778) -- (1.5215,0.7778);
    \draw[thick, blue] (1.5215,0.7778) -- (1.5215,0.7531);
    \draw[thick, blue] (1.5295,0.7531) -- (1.5295,0.6667);
    \draw[thick, blue] (1.5375,0.6667) -- (1.5375,0.7778);
    \draw[thick, blue] (1.5455,0.7778) -- (1.5616,0.7778);
    \draw[thick, blue] (1.5616,0.7778) -- (1.5616,0.6667);
    \draw[thick, blue] (1.5776,0.6667) -- (1.5776,0.7778);
    \draw[thick, blue] (1.5936,0.7778) -- (1.5936,0.7407);
    \draw[thick, blue] (1.6096,0.7407) -- (1.6096,0.7490);
    \draw[thick, blue] (1.6176,0.7490) -- (1.6176,0.6667);
    \draw[thick, blue] (1.6256,0.6667) -- (1.6657,0.6667);
    \draw[thick, blue] (1.6657,0.6667) -- (1.6657,1.0000);
    \draw[thick, blue] (1.6737,1.0000) -- (1.7457,1.0000);
    \draw[thick, blue] (1.7457,1.0000) -- (1.7457,0.6667);
    \draw[thick, blue] (1.7538,0.6667) -- (1.8258,0.6667);
    \draw[thick, blue] (1.8258,0.6667) -- (1.8258,1.0000);
    \draw[thick, blue] (1.8338,1.0000) -- (1.8739,1.0000);
    \draw[thick, blue] (1.8739,1.0000) -- (1.8739,0.7407);
    \draw[thick, blue] (1.8819,0.7407) -- (1.8819,0.7778);
    \draw[thick, blue] (1.8979,0.7778) -- (1.8979,0.7407);
    \draw[thick, blue] (1.9059,0.7407) -- (1.9059,0.6667);
    \draw[thick, blue] (1.9139,0.6667) -- (1.9139,1.0000);
    \draw[thick, blue] (1.9299,1.0000) -- (1.9299,0.6667);
    \draw[thick, blue] (1.9379,0.6667) -- (1.9540,0.6667);
    \draw[thick, blue] (1.9540,0.6667) -- (1.9540,1.0000);
    \draw[thick, blue] (1.9620,1.0000) -- (1.9620,0.6667);
    \draw[thick, blue] (1.9700,0.6667) -- (1.9860,0.6667);
    \draw[thick, blue] (1.9860,0.6667) -- (1.9860,0.7778);
    \draw[thick, blue] (1.9940,0.7778) -- (1.9940,0.0000);
    \draw[thick, blue] (2.0020,0.0000) -- (2.6667,0.0000);
    \draw[thick, blue] (2.6667,0.0000) -- (2.6667,1.0000);
    \draw[thick, blue] (2.6747,1.0000) -- (3.9960,1.0000);
    \draw[thick, blue] (3.9960,1.0000) -- (3.9960,0.0000);
    \draw[thick, blue] (4.0040,0.0000) -- (5.3333,0.0000);
    \draw[thick, blue] (5.3333,0.0000) -- (5.3333,1.0000);
    \draw[thick, blue] (5.3413,1.0000) -- (5.9980,1.0000);
    \draw[thick, blue] (5.9980,1.0000) -- (5.9980,0.7407);
    \draw[thick, blue] (6.0060,0.7407) -- (6.0060,0.7778);
    \draw[thick, blue] (6.0140,0.7778) -- (6.0300,0.7778);
    \draw[thick, blue] (6.0300,0.7778) -- (6.0300,0.6667);
    \draw[thick, blue] (6.0380,0.6667) -- (6.0380,0.7778);
    \draw[thick, blue] (6.0460,0.7778) -- (6.0621,0.7778);
    \draw[thick, blue] (6.0621,0.7778) -- (6.0621,0.6667);
    \draw[thick, blue] (6.0781,0.6667) -- (6.0781,0.7778);
    \draw[thick, blue] (6.0861,0.7778) -- (6.0861,0.7531);
    \draw[thick, blue] (6.0941,0.7531) -- (6.0941,0.7407);
    \draw[thick, blue] (6.1101,0.7407) -- (6.1101,0.7531);
    \draw[thick, blue] (6.1181,0.7531) -- (6.1181,0.6667);
    \draw[thick, blue] (6.1261,0.6667) -- (6.1662,0.6667);
    \draw[thick, blue] (6.1662,0.6667) -- (6.1662,0.7778);
    \draw[thick, blue] (6.1742,0.7778) -- (6.2462,0.7778);
    \draw[thick, blue] (6.2462,0.7778) -- (6.2462,0.6667);
    \draw[thick, blue] (6.2543,0.6667) -- (6.3263,0.6667);
    \draw[thick, blue] (6.3263,0.6667) -- (6.3263,0.7778);
    \draw[thick, blue] (6.3343,0.7778) -- (6.3744,0.7778);
    \draw[thick, blue] (6.3744,0.7778) -- (6.3744,0.7503);
    \draw[thick, blue] (6.3824,0.7503) -- (6.3824,0.7531);
    \draw[thick, blue] (6.3984,0.7531) -- (6.3984,0.7407);
    \draw[thick, blue] (6.4144,0.7407) -- (6.4144,0.7778);
    \draw[thick, blue] (6.4304,0.7778) -- (6.4304,0.7407);
    \draw[thick, blue] (6.4384,0.7407) -- (6.4545,0.7407);
    \draw[thick, blue] (6.4545,0.7407) -- (6.4545,0.7778);
    \draw[thick, blue] (6.4625,0.7778) -- (6.4625,0.7490);
    \draw[thick, blue] (6.4705,0.7490) -- (6.4705,0.7407);
    \draw[thick, blue] (6.4785,0.7407) -- (6.4945,0.7407);
    \draw[thick, blue] (6.4945,0.7407) -- (6.4945,0.6667);
    \draw[thick, blue] (6.5025,0.6667) -- (6.6627,0.6667);
    \draw[thick, blue] (6.6627,0.6667) -- (6.6627,1.0000);
    \draw[thick, blue] (6.6707,1.0000) -- (6.9990,1.0000);
    \draw[thick, blue] (6.9990,1.0000) -- (6.9990,0.6667);
    \draw[thick, blue] (7.0070,0.6667) -- (7.3273,0.6667);
    \draw[thick, blue] (7.3273,0.6667) -- (7.3273,1.0000);
    \draw[thick, blue] (7.3353,1.0000) -- (7.4955,1.0000);
    \draw[thick, blue] (7.4955,1.0000) -- (7.4955,0.7778);
    \draw[thick, blue] (7.5115,0.7778) -- (7.5115,0.6667);
    \draw[thick, blue] (7.5195,0.6667) -- (7.5195,0.7407);
    \draw[thick, blue] (7.5275,0.7407) -- (7.5275,0.6667);
    \draw[thick, blue] (7.5355,0.6667) -- (7.5355,0.7778);
    \draw[thick, blue] (7.5435,0.7778) -- (7.5596,0.7778);
    \draw[thick, blue] (7.5596,0.7778) -- (7.5596,0.6667);
    \draw[thick, blue] (7.5756,0.6667) -- (7.5756,0.7778);
    \draw[thick, blue] (7.5916,0.7778) -- (7.5916,0.7531);
    \draw[thick, blue] (7.5996,0.7531) -- (7.5996,0.7778);
    \draw[thick, blue] (7.6156,0.7778) -- (7.6156,0.7407);
    \draw[thick, blue] (7.6236,0.7407) -- (7.6236,0.6667);
    \draw[thick, blue] (7.6316,0.6667) -- (7.6637,0.6667);
    \draw[thick, blue] (7.6637,0.6667) -- (7.6637,1.0000);
    \draw[thick, blue] (7.6717,1.0000) -- (7.7437,1.0000);
    \draw[thick, blue] (7.7437,1.0000) -- (7.7437,0.6667);
    \draw[thick, blue] (7.7518,0.6667) -- (7.8318,0.6667);
    \draw[thick, blue] (7.8318,0.6667) -- (7.8318,1.0000);
    \draw[thick, blue] (7.8398,1.0000) -- (7.8719,1.0000);
    \draw[thick, blue] (7.8719,1.0000) -- (7.8719,0.6667);
    \draw[thick, blue] (7.8799,0.6667) -- (7.8799,0.7778);
    \draw[thick, blue] (7.8879,0.7778) -- (7.9039,0.7778);
    \draw[thick, blue] (7.9039,0.7778) -- (7.9039,0.6667);
    \draw[thick, blue] (7.9119,0.6667) -- (7.9119,1.0000);
    \draw[thick, blue] (7.9199,1.0000) -- (7.9359,1.0000);
    \draw[thick, blue] (7.9359,1.0000) -- (7.9359,0.6667);
    \draw[thick, blue] (7.9520,0.6667) -- (7.9520,1.0000);
    \draw[thick, blue] (7.9680,1.0000) -- (7.9680,0.7778);
    \draw[thick, blue] (7.9760,0.7778) -- (7.9760,1.0000);
    \draw[thick, blue] (7.9840,1.0000) -- (8.0000,1.0000);
    
\end{tikzpicture}

\begin{prop}
\label{random}
With probability $1$, if $k$ is sufficiently large, then
for all 
$$n\in [3N_k/2,6N_k)\setminus \{4N_k\},$$
we have $\rept{B(k)}n \ge h(n)$ and likewise for $C(k)$.
\end{prop}

\begin{proof}
We note that $B(k)$ contains the following subsets: 
$$X_1\!\left(\frac{N}{24},\frac{3N}{64}\right), X_1\!\left(\frac{2N}{3},\frac{3N}{4}\right),  X_1\!\left(\frac{4N}{3}, 2N\right), X_1\!\left(\frac{8N}{3}, 3N\right),
X_1(4N-X_3(N/4,3N/4)).
$$
Indeed, the first belongs to $B_{k-3}$ and the second to $B_{k-1}$. For the last inclusion, we are using the fact that $N_{i+1}-G_i\subset [0,N/4] \cup [3N/4,N]$ for all $i<k$.

There exists $\epsilon > 0$ such that the following statements are all true with probability at least $1-(1-\epsilon)^N$.
By Lemma~\ref{sum} applied to $i=1$,
$$I\in \left\{\left(\frac{N}{24},\frac{3N}{64}\right), \left(\frac{2N}{3},\frac{3N}{4}\right),  \left(\frac{4N}{3}, 2N\right), \left(\frac{8N}{3}, 3N\right)\right\},$$
and $J = (4N/3,2N)$, every integer in $(67N/48,95N/24) \cup (25N/6,29N/6)$ has at least $h(N)$ representations over $B(k)$.
Applying Lemma~\ref{intersection} with $I=(2N/3,3N/4)$, every integer in the range $[95N/24,25N/6]$ other than $4N$ has at least
$h(N)$ representations over $B(k)$. Applying the same lemma to $I=(4N/3,2N)$, every integer in $[29N/6, 45N/8]$
has at least $h(N)$ representations over $B(k)$. Lemma~\ref{sum} with $I = J = (8N/3,3N)$ deals with $[45N/8,143N/24]$.
Finally, applying Lemma~\ref{intersection} with $I=(8N/3,3N)$, we cover $[143N/24, 6N]$.
By the Borel-Cantelli lemma, these statements hold for all sufficiently large $k$ with probability $1$.

Note that in all cases the sums being considered have summands whose ratio lies in $[1/100,100]$.
\end{proof}

\begin{proof}[Proof of Theorem~\ref{construction}]
Looking at \eqref{B_k} and \eqref{C_k}, it is clear that Properties 1 and 4 hold.
We have that $B_k$ and $C_k$ are disjoint from one another (because $X_1$ and $X_2$ are so), and they are disjoint from $B(k-1)$ and $C(k-1)$ by size considerations.
Since $G_k,H_k\subset A(k-1)$ by hypothesis on $\S$, Property 2 holds as well. (Note that $4N-G_k, 4N-H_k \subset (3N,4N)$ are disjoint from $A(k-1)$.)
By Proposition~\ref{random}, Property 3 holds with probability $1$.

\end{proof}

We acknowledge the assistance of Claude Sonnet 4.5 with the graphics in this paper and ChatGPT 5.2 Pro with proofreading.

\end{document}